\documentclass[reqno,a4paper,12pt]{amsart}
\usepackage{a4wide}
\usepackage{graphicx}
\usepackage{tikz-cd} 
\title[Quartic del Pezzo surfaces over $\F_p(t)$ without quadratic points]{Quartic del Pezzo surfaces over $\F_p(t)$ without quadratic points}

\author{Giorgio Navone}
\address{Giorgio Navone, Department of Mathematics, University College London, 25 Gordon St, WC1H 0AY, London, United Kingdom }
\email{giorgio.navone.22@ucl.ac.uk}
\urladdr{https://sites.google.com/view/giorgio-navone}

\author{Katerina Santicola}
\address{Katerina Santicola, Department of Mathematical Sciences, University of Bath, Claverton Down, BA2 7AY, Bath, United Kingdom}
\email{ks3056@bath.ac.uk}
\urladdr{https://sites.google.com/view/katerinasanticola}

\author{Harry C. Shaw}
\address{Harry C. Shaw, Department of Mathematical Sciences, University of Bath, Claverton Down, BA2 7AY, Bath, United Kingdom}
\email{hcs50@bath.ac.uk}
\urladdr{https://sites.google.com/view/harry-c-shaw}

\author{Haowen Zhang} \address{Haowen Zhang, Mathematical Institute, Leiden University, Einsteinweg 55, 2333 CC Leiden, The Netherlands} \email{h.zhang@math.leidenuniv.nl} \urladdr{https://hzhang.perso.math.cnrs.fr/}

\date{\today}

\usepackage{color}
\usepackage{mathtools}
\usepackage{musicography}
\usepackage{amsmath}
\definecolor{darkgreen}{rgb}{0,0.5,0}

\usepackage[backref]{hyperref} 
\definecolor{darkred}{rgb}{0.5,0,0}
\definecolor{darkgreen}{rgb}{0,0.5,0}
\definecolor{darkblue}{rgb}{0,0,0.5}
\hypersetup{linkcolor=darkgreen,filecolor=darkgreen,urlcolor=darkgreen,citecolor=darkgreen,colorlinks}
\usepackage{enumitem}

\newcounter{starcond}

\newcommand{\Z}{\mathbb{Z}}
\newcommand{\F}{\mathbb{F}}
\newcommand{\Q}{\mathbb{Q}}

\newcommand{\PP}{\mathbb{P}}

\newcommand{\OO}{\mathcal{O}}
\newcommand{\Proj}{\mathbb{P}}
\newcommand{\oo}{\mathcal{O}}
\usepackage{tcolorbox,amssymb}

\newtheorem{theorem}{Theorem}[section]

\newtheorem{lemma}[theorem]{Lemma}
\newtheorem{corollary}[theorem]{Corollary}

\newtheorem{proposition}[theorem]{Proposition}

\theoremstyle{remark}
\newtheorem{remark}[theorem]{Remark}

\numberwithin{equation}{subsection}

\theoremstyle{definition}
\newtheorem{example}[theorem]{Example}

\usepackage[OT2,T1]{fontenc}
\DeclareSymbolFont{cyrletters}{OT2}{wncyr}{m}{n}
\DeclareMathSymbol{\Sha}{\mathalpha}{cyrletters}{"58}

\DeclareMathOperator{\inv}{inv}

\DeclareMathOperator{\ev}{ev}
\DeclareMathOperator{\Pic}{Pic}

\DeclareMathOperator{\Spec}{Spec}
\DeclareMathOperator{\et}{\acute{e}t}
\DeclareMathOperator{\Br}{Br}
\DeclareMathOperator{\Gal}{Gal}
\DeclareMathOperator{\disc}{disc}

\begin{document}

\maketitle
	
\subsection*{Abstract}
We construct an infinite family of quartic del Pezzo surfaces over $\F_p(t)$ with no quadratic points, for all primes $p\neq 2$. This answers a question of Colliot--Th\'el\`ene, Creutz and Viray in the negative, which asks whether every quartic del Pezzo surface has quadratic points over $C_2$ fields. We exhibit a Brauer--Manin obstruction on the variety parametrising lines associated to the quartic del Pezzo surface.
	
\setcounter{tocdepth}{1}
\tableofcontents

\section{Introduction}

Given a variety $X$ over a field $k$, a fundamental question of arithmetic geometry is whether $X$ has a $k$-rational point. If not, it is natural to ask what is the smallest extension $L/k$ (with respect to degree) such that $X_L$ has an $L$-point?  \par
In this paper we study this question for quartic del Pezzo surfaces over $\F_p(t)$, for $p>2$ a prime. Namely these are the surfaces $X\subseteq \PP^4_{\F_p(t)}$ given by the smooth, complete intersection of two quadrics. Quartic del Pezzo surfaces over local fields (including the $C_2$ fields $\mathbb F_p(\!(t)\!)$) always admit quadratic points (\cite[Theorem 1.2]{CreutzViray2023}). Creutz and Viray (\cite[Page~1412]{CreutzViray2023}) and Colliot--Thélène (\cite[Remarque~3.12]{colliot}) have raised the question of whether every quartic del Pezzo surface over a $C_2$ field admits a quadratic point. We answer this question in the negative by considering the $C_2$ fields $\F_p(t)$:
\begin{theorem}\label{thm:maindp4}
Let $p>2$ be a prime. There exist  infinitely many non-isomorphic quartic del Pezzo surfaces $X$ over $\F_p(t)$ such that $X(K)=\varnothing$ for all quadratic extensions $K/\F_p(t)$. 
\end{theorem} 

The hypothesis $p>2$ is necessary, since by \cite[Theorem~1.2.(3)]{CreutzViray2023} every quartic del Pezzo surface over a global function field of characteristic $2$ has a quadratic point. Furthermore, in \cite[Theorem~1]{creutz2024quartic} the authors proved the existence of quartic del Pezzo surfaces over $\Q$ which admit no quadratic point. In light of these results, we obtain the following:

\begin{corollary}\label{cor:1.2}
Let $p$ be a prime or $0$. Then there exists a quartic del Pezzo surface over a global field of characteristic $p$ not admitting a quadratic point if and only if $p\neq 2$.
\end{corollary}

For a quartic del Pezzo surface of index 1, the theorems of Amer--Brumer and Springer (\cite{Amer1976, Brumer1978, SPRINGER1956238}) imply the existence of a rational point.  The index of a quartic del Pezzo surface over global fields divides 2, by \cite[Theorem 1.1]{CreutzViray2023} over number fields, and by \cite[Corollary 6.5]{CreutzViray2023} and \cite[Theorem 1.4]{Tian2025LocalGlobal} over global function fields. Therefore Corollary \ref{cor:1.2} shows that, over global fields, index 2 does not always imply the existence of a quadratic point.

By the Amer--Brumer theorem (\cite{Amer1976, Brumer1978}), every quartic del Pezzo surface over a $C_1$ field has a rational point. On the other hand, by \cite[Theorem 7.6]{CreutzViray2023} there exists a quartic del Pezzo surface $X$ over a $C_3$ field of characteristic 0 not admitting a quadratic point. Together with these results, we obtain a complete answer to this question over $C_i$ fields:

\begin{corollary}
Let $i \geq 0$ be an integer. There exists a quartic del Pezzo surface over a $C_i$ field not admitting a quadratic point if and only if $i>1$.
\end{corollary}

We obtain Theorem~\ref{thm:maindp4} by computing a Brauer--Manin obstruction to the Hasse principle on the variety of lines associated to the quartic del Pezzo surface, similarly to the method of \cite{creutz2024quartic}. Unlike in \cite{creutz2024quartic}, since we work uniformly over all $\F_p(t)$, away from $p=3,7,11$, we are unable to explicitly determine the primes of bad reduction of the surface. In order to circumvent this issue, we utilise an argument involving quadratic reciprocity and Chebotarev's density theorem in the global function field setting. Furthermore, unlike in the number field case, we cannot use a computer search to look for candidate quartic del Pezzo surfaces, so instead we manually construct such surfaces.

\subsection{Outline of the paper}
In Section \ref{sec:prelim} we recall preliminaries on the Brauer--Manin obstruction over global function fields and the arithmetic of quartic del Pezzo surfaces. \par
In Section \ref{section:thevarietyG} we introduce the variety $\mathcal{G}$ associated to the quartic del Pezzo surface, a key object whose rational points correspond to quadratic points on the quartic del Pezzo surface. \par
In Section \ref{section:pneq3} we present an infinite family of quartic del Pezzo surfaces defined over $\F_p(t)$ for $p>3$. We address the case $p=3$ separately in Section \ref{sec:p=3}.\par
In Section \ref{sec:proofoftheorem3} we prove Theorem \ref{thm:maindp4} by combining the arguments from Section \ref{section:pneq3} and \ref{sec:p=3}. \par
In Section \ref{sec:howwechose} we give some background on how we constructed these families.

\subsection*{Acknowledgments}
The authors would like to express their gratitude towards Daniel Loughran and Rachel Newton for their continued support throughout the process of writing the paper. The authors would also like to thank Harmeet Singh for his helpful discussion regarding the wreath product, and thank Brendan Creutz for helpful feedback on an earlier draft. 

The authors were supported by a Heilbronn Focused Research Grant, through the UKRI/EPSRC ‘Additional Funding Programme for Mathematical Sciences’.
The first author was also supported by the Engineering and Physical Sciences Research Council [EP/S021590/1], the EPSRC Centre for Doctoral Training in Geometry and Number Theory (The London School of Geometry and Number Theory), University College London. 
The second author was supported by a doctoral studentship from the Heilbronn Institute
for Mathematical Research, and an EPSRC postdoctoral pathway fellowship. 

\section{Preliminaries}\label{sec:prelim}

\subsection{Notation}

Given a global field $k$, let $\Omega_k$ denote the set of its places and $\OO_k$ its ring of integers. Given a place $\nu\in\Omega_k$, let $k_\nu$ denote the completion of $k$ at $\nu$ and let $\OO_\nu$ denote its ring of integers.\par
Given a field $k$, fix an algebraic closure $\overline{k}$ of $k$. Then given a finite separable extension $L/k$, we let $L^{\Gal}$ denote the Galois closure of $L$ in $\overline{k}$.

If $G$ is an abelian group and $l$ a prime, let $G[l]$ denote the $l$-torsion subgroup and let $G\{l\}$ denote the $l$-primary subgroup.

Let $X$ be a variety over a field $k$. We say $X$ is \textbf{nice} if it is smooth, projective, and geometrically integral. 
    
\subsection{The Brauer--Manin obstruction} Let $k$ be a global field.
Denote by $X(\mathbb{A}_k)$ the set of adelic points on $X$. There is a natural pairing 
\begin{align*}
X(\mathbb{A}_k)\times \Br \left(X\right)&\rightarrow \Q/\Z \\
((x_\nu)_{\nu}, \mathcal{A})&\mapsto \sum_{{\nu}\in\Omega_k}\inv_{\nu}(\ev_\mathcal{A} (x_{\nu}))
\end{align*}
known as the Brauer--Manin pairing, 
where $\ev_\mathcal{A}: X(k_{\nu})\rightarrow \Br \left(k_{\nu}\right)$ is the evaluation map induced by $\mathcal{A}\in\Br \left(X\right)$, and $\inv_{\nu}\colon \Br \left(k_{\nu}\right)\hookrightarrow \Q/\Z$ is the local invariant map from class field theory \cite[Definition 13.1.7]{ColliotThelene2021}. We denote by $X(\mathbb{A}_k
)^{\Br}$ the set of adelic points which are orthogonal to $\Br \left(X\right)$ with respect to the Brauer--Manin pairing, that is
\begin{equation*}
X(\mathbb{A}_k
)^{\Br}\coloneqq \left\{(x_\nu)_{\nu}\in X(\mathbb{A}_k): \sum_{{\nu}\in\Omega_k}\inv_{\nu}(\ev_\mathcal{A} (x_{\nu}))=0 \text{ for all } \mathcal{A}\in \Br(X)\right\}. 
\end{equation*} 
The set of rational points $X(k)$ is contained in $X(\mathbb{A}_k
)^{\Br}$ by the Albert-Brauer-Hasse-Noether exact sequence (\cite[Theorem 13.1.8]{ColliotThelene2021}).
If $X(\mathbb A_k)\neq \varnothing$ but $X(\mathbb A_k)^{\Br}=\varnothing,$ we say that there is a \textbf{Brauer--Manin obstruction to the Hasse principle}.

Now we discuss good reduction of the Brauer--Manin obstruction, especially over global function fields. If $X$ is a nice variety over a global field, then similarly to \cite[Theorem 13.3.15]{ColliotThelene2021} we can produce conditions under which a place of good reduction has constant local evaluation maps. Although \cite[Theorem 13.3.15]{ColliotThelene2021} is presented for number fields, the statement also holds for a global function field $k$ if we assume the characteristic of $k$ is coprime to the order of $\Br \left(X\right)/\Br_0\left(X\right)$. Let
\[
\Br (X)[p']:=\{\mathcal{A}\in \Br(X) \;\big|\; p^{i}\mathcal{A} \neq 0 \text{ for all } i \ge 1   \}=\bigoplus_{\ell \neq p}\Br(X)\{\ell\}.
\]
	
\begin{proposition}
\label{proposition: good reduction at a place imples constant inv map}
Let $k$ be a global field of characteristic $p$. Let $X$ be a nice variety over $k$ with torsion-free $\Pic \overline X$. Suppose $X$ has good reduction at $\nu$. Then for all $\mathcal{A}\in \Br (X)[p']$, the evaluation map $\ev_\mathcal{A}: X(k_{\nu})\rightarrow \Br (k_{\nu})$ is constant.
\end{proposition}

\begin{proof}
Since $\mathcal{A}\in \Br(X)[p']$, we may reduce to the case $\mathcal{A}\in \Br(X)\{\ell\}$ for some $\ell\neq p$. By definition of good reduction at $\nu$ there is a smooth and proper $\mathcal O _{\nu}$-scheme $\mathcal X$ with generic fiber $ X_{\nu}=X\times_k \Spec k_{\nu}$. 
Let $\overline{\mathcal X}_0$ be the closed geometric fiber of $\mathcal X\rightarrow\Spec\mathcal O_{\nu}$. By \cite[Cor. VI.4.2 and Cor. VI.4.3]{milne1980etale}, we have $\text{H}_{\text{\'et} } ^1(\overline{\mathcal X}_0,\Z/\ell)\simeq \text{H}_{\text{\'et}} ^1(\overline{X}_{\nu},\Z/\ell)\simeq \text{H}_{\text{\'et}} ^1(\overline{\mathcal X},\Z/\ell)$. The Kummer exact sequences gives $\text{H}_{\et} ^1(\overline{\mathcal X},\mu_\ell)\simeq\Pic(\overline X)[\ell]$ which vanishes by our assumption that $\Pic(\overline X)$ is torsion free. Since $\ell \neq \text{char}(k)=p$, we can apply \cite[Proposition 10.4.3]{ColliotThelene2021} and write the image of $\mathcal{A}$ in $\Br X_{\nu}$ as $\mathcal{A}_0+\mathcal{A}_1$, with $\mathcal{A}_0\in\Br \mathcal X$ and $\mathcal{A}_1\in \Br k_{\nu}$. By the properness of $\mathcal X\rightarrow \Spec\mathcal O_{\nu}$, any $k_{\nu}$-point of $X_{\nu}$ extends to an $\mathcal O_{\nu}$-point of $\mathcal X$, and thus the evaluation of $\mathcal{A}_0$ lands in $\Br \mathcal O_{\nu}$. Since $\Br \oo_{\nu}\cong \Br \F_p$ (\cite[Theorem 3.4.2(i)]{ColliotThelene2021}), and $\Br \F_p =0$ since $\F_p$ is $C_1$ (\cite[Theorem 1.2.13]{ColliotThelene2021}), we see that $\ev_{\mathcal{A}_0}$ vanishes. Therefore the evaluation map takes the constant value $\mathcal{A}_1\in\Br (k_{\nu})$. 
\end{proof}
	
\begin{remark}
We require fewer conditions than that of \cite[Theorem 13.3.15]{ColliotThelene2021} since we ask for the weaker conclusion: that only the elements in $\Br(X)[p']$ have constant evaluation. 
\end{remark}

\subsection{Quartic del Pezzo surfaces}
A \textbf{del Pezzo surface} $X$ is a nice surface over $k$ whose anticanonical sheaf is ample.  
The degree of $X$ is the self-intersection number
\(
d := K_X^2.
\)
Every quartic del Pezzo (a del Pezzo surface of degree 4) arises as the complete intersection of two quadrics in $\PP^4$. 
	
Let $X:\{Q_0=Q_{\infty}=0\}\subset \PP^4$ be a quartic del Pezzo over a field $k$. Let $M_0$ and $M_{\infty}$ be the symmetric matrices corresponding to the quadratic forms $Q_0$ and $Q_{\infty}$. The \textbf{characteristic polynomial} of $X$ is
\[
	f(x):=\det(M_0+xM_{\infty})\in k[x].
\]
Since $X$ is smooth, $f$ is a separable polynomial of degree 5 (\cite[Proposition 3.26]{Wittenberg2007}). 
Let $f$ factor as
\[
    f=f_1\cdots f_r \in k[x],
\]
where each $f_i$ is an irreducible factor of $f$. Let $k(\theta_i)\cong k[x]/f_i(x)$ be the field extension associated to $f_i$. Let $q_i$ be the rank 4 quadric corresponding to $M_0+\theta_iM_{\infty}$ (\cite[Section 3.3]{Wittenberg2007}), and let $H_i\subseteq \PP^4$ be any hyperplane  not containing the vertex of $q_i$.
For each $i$, we define $\epsilon_i$ to be the invariant
\[ 
\epsilon_i:=\disc \left(q_i \cap H_i\right)\in k(\theta_i)^{\times}/k(\theta_i)^{\times 2},
\]    
where the discriminant of $q_i\cap H_i$ is the determinant of the corresponding rank 4 matrix. The definition of $\epsilon_i \in k(\theta_i)^{\times}/k(\theta_i)^{\times 2}$ does not depend on the choice of $H_i$ (see \cite[Section 3.4.1]{Wittenberg2007}).

\section{The variety \texorpdfstring{$\mathcal{G}$}{}}\label{section:thevarietyG}
Let $X$ be a quartic del Pezzo surface over a field $k$ of odd characteristic.  
Let $\mathcal{Q}\subset \PP^4\times \PP^1$ be the smooth pencil of quadrics corresponding to $X$. Let $Q_t$ be the fibre of $\mathcal{Q}$ over $t\in \PP^1$. Consider the fourfold $\pi: \mathcal{G}\rightarrow \PP^1$ parametrising lines on quadrics in the pencil $\mathcal{Q}$ \cite[Section 1.2]{reid1972complete}. The rational points on $\mathcal{G}$ correspond to quadratic points on $X$ by \cite[Proposition 4.4(1)]{CreutzViray2023}, which we reprove in the following lemma:
	
\begin{lemma}\label{lem:quadpointrationalline}
Let $X:\{Q_0=Q_\infty=0\}$ be a quartic del Pezzo surface, then $X$ has a quadratic point if and only if $Q_t$ contains a rational line for some $t \in \mathbb{P}^1(k)$. Equivalently, $X$ has a quadratic point if and only if $\mathcal{G}(k)\neq \varnothing$.
\end{lemma}       

\begin{proof}
Given a quadratic point $P$ on $X$, suppose first that $P$ and its conjugate $\Bar{P}$ are distinct. Consider the rational line $l\subset \PP^4_k$ passing through $P$ and $\Bar{P}$. It is defined over $k$ since it is fixed by Galois action. Now take any third point $P'\in \PP^4(k)$ on $l$, distinct from $P$ and $\bar{P}$ since it is rational. Solve for $t_0$ such that $Q_0(P')+t_0Q_{\infty}(P')=0$. Then the quadric $Q_{t_0}:Q_0+t_0Q_{\infty}$ satisfies $\#(Q_{t_0} \cap l) \geq 3$, and therefore $Q_{t_0}$ must contain $l$ by Bezout's theorem. 
        
Now suppose $P,\Bar{P}$ are the same point, so $P$ is rational. Then take any rational line $\ell\subset \PP^4_k$ through $P$ tangent to $X$. Take a second point $P'\in \PP^4(k)$ distinct from $P$, and solve for $t_0$ such that $Q_0(P')+t_0Q_\infty(P')=0$. Then the quadric $Q_{t_0}:Q_0+t_0Q_\infty$ must again contain $\ell$ by Bezout's theorem, since $\ell$ intersects $Q_{t_0}$ at $P$ with multiplicity two. 
		
On the other hand, given a rational line $l$ in $Q_t$, the intersection of $l$ with $Q_0$ or $Q_{\infty}$ consists of two points, fixed by Galois action, and so $X$ has a quadratic point.
\end{proof}
	
We therefore wish to understand the Brauer--Manin obstruction on $\mathcal{G}$. The following proposition is known to experts over number fields, the case of global function fields follows similarly:

\begin{proposition}\label{Gisnice} Let $X$ and $\mathcal{G}$ be as above, defined over a global function field $k$ of odd characteristic. Then for all primes $\nu$ of good reduction of $\mathcal{G}$, we have that $\ev_\mathcal{A}: \mathcal{G}(k_{\nu})\rightarrow \Br (k_{\nu})$ is constant for all $\mathcal{A}\in \Br \mathcal{G}$.
\end{proposition}

\begin{proof}
Since $\mathcal{G}$ is a closed subscheme of $\PP^1\times \text{Gr}(2,5)$ (\cite[Definition 1.9]{reid1972complete}), it is a projective variety. 
Furthermore, by \cite[Proposition 2.1]{reid1972complete} $\mathcal{G}$ is smooth, and hence $\mathcal{G}$ is a nice variety. 
It is birational to $\text{Sym}^2(X)$ by \cite[Proposition 4.4 (i)]{CreutzViray2023}, which is simply connected, and so $\mathcal{G}$ is simply connected.
Therefore  $\Pic(\bar{\mathcal{G}})$ is torsion-free. By \cite[Corollary~5.2]{CreutzViray2023} we have that $\Br\mathcal{G}/\Br_0\mathcal{G}$ is $2$-torsion, hence by Proposition~\ref{proposition: good reduction at a place imples constant inv map} the result follows.
\end{proof}
	
\section{Existence for \texorpdfstring{$p>3$}{}}\label{section:pneq3}

In this section, we fix a prime $p>3$ and construct an infinite family of quartic del Pezzo surfaces without quadratic points over $k=\F_p(t)$. 

\begin{lemma}\label{lem:choiceofalpha}\label{lem: constructing alpha}
For all primes $p>3$, there exists an $\alpha\in\F_p^{\times}$ such that the following hold:
\begin{enumerate}
    \item \label{alpha is not a square} $\alpha \not\in \F_p^{\times 2}$;
    \item \label{alpha+1 is a square} $\alpha +1\in \F_p^{\times2}$;
    \item \label{3alpha-1 is nonzero} $3\alpha -1\in \F_p^\times$.
\end{enumerate}
Furthermore, if $p\notin\{7,11\}$, then there exists $\alpha\in\F_p^\times$ which additionally satisfies 
\begin{enumerate}
\setcounter{enumi}{3}
    \item \label{3alpha-1 is not a square} $3\alpha-1\notin\F_p^{\times 2}$.
\end{enumerate}
\end{lemma} 

\begin{proof}
Let $\chi$ be the character associated to the Legendre symbol modulo $p$. Let \[N:=\#\{x \in \F_p: \chi(x)=-1, \chi(x+1)=1 \text{ and } \chi(3x-1)=-1\}. \]
         
Then we can estimate $N$ in terms of polynomial sums of $\chi$, i.e. \[N + \frac{3}{2} \geq \sum_{x \in \F_p} \frac{1-\chi(x)}{2} \frac{1+\chi(x+1)}{2}\frac{1-\chi(3x-1)}{2}, \] where $\frac{3}{2}$ is the maximal contribution arising from the vanishing of $\chi$. Expanding the sum and recalling that linear sums vanish, we obtain 
\[ N+\frac{3}{2} \geq \frac{1}{8}\left(p+\sum_{x \in \F_p} \chi(p_1(x)) + \chi(p_2(x))+\chi(p_3(x))+ \chi(p_4(x))\right),\] 
where $p_1, p_2$ and $p_3$ are polynomials of degree 2 and $p_4$ is a polynomial of degree 3. Now, the Weil bound on character sums implies that $N +\frac{3}{2} \geq \frac{1}{8}(p-5\sqrt{p})$ or, equivalently, $N \geq \frac{1}{8}(p-5\sqrt{p}-12)$. For $p \geq 46$, the inequality implies $N>0$ which provides at least one choice for $\alpha$. For $p<46$, we explicitly find $\alpha$ satisfying our conditions \ref{alpha is not a square}, \ref{alpha+1 is a square}, \ref{3alpha-1 is nonzero} and \ref{3alpha-1 is not a square} except for $p=7$, where we choose $\alpha=3$, and $p=11$, where we choose $\alpha=8$, satisfying only the first three properties. 
\end{proof}

\subsection{The family of quartic del Pezzo surfaces}

Fix a choice of $\alpha\in\F_p^{\times}$ given by Lemma~\ref{lem: constructing alpha}. For a given $d\in k^{\times}$, we define $X_{d}=\{Q_0=Q_\infty=0\}\subseteq\PP^4_k$ to be the surface given by intersection of the two quadrics
\begin{align}
Q_0:&=d(u_0^2-tu_1^2)+(t^2+t^3)u_2^2+2\alpha tu_2u_3+2t^2u_3u_4,\label{q0}\\
Q_{\infty}:&=-2du_0u_1+\frac{\alpha}{t}u_2^2+ t^2u_3^2+2\alpha u_2u_4+\frac{(3\alpha -1)t}{\alpha+1 }u_4^2. \label{qinfty}
\end{align}
This defines a smooth, complete intersection. The surfaces in Theorem~\ref{thm:maindp4} arise from this family for $p>3$ (see Section~\ref{sec:howwechose} for how we constructed this family). The characteristic polynomial of $X_{d}$ is
\[
	f(x):=c\left(x^2 + t\right)f_3(x),
\]
where
\[
	f_3(x):=x^3 - \frac{t^3(t+1)(3\alpha -1)}{\alpha (\alpha -1)^2}x^2 + tx + \frac{t^4(t+1)(\alpha +1)}{\alpha (\alpha -1)^2},
\]
and 
\[
	c=\frac{d^2\alpha (\alpha -1)^2t^2}{\alpha +1}.
\]
Let $\theta\in k[x]/f_3(x)$ denote the image of $x$ under the canonical map $k[x]\rightarrow k[x]/f_3(x)$, and let $\epsilon_2,\epsilon_3$ denote the discriminants of the two rank four quadrics. Up to squares, they are of the form
\begin{equation*}\label{eq: definition of ep2 and ep3}   
\epsilon_2:=d\alpha t(t+1)\in k(\sqrt{-t})/k(\sqrt{-t})^{\times2},\qquad \epsilon_3:=-\alpha(t+1)\theta \in k(\theta)/k(\theta)^{\times2}.
\end{equation*}

We compute the discriminant $\Delta$ of $f(x)$, and the discriminant $\Delta_{f_3}$ of $f_3$:
\[
	\Delta:=\frac{-2^{12}\alpha ^4t^{36}(t+1)^4d^{16}\xi}{(\alpha +1)^8}, \qquad	\Delta_{f_3}:=\frac{2^2t^3\xi}{\alpha^4(\alpha-1)^8},
\]
where
\[
	\xi:=t^{10}(\alpha +1)(3\alpha -1)^3(t+1)^4-2\alpha ^2t^5(\alpha -1)^4(9\alpha^2+12\alpha+1)(t+1)^2-\alpha ^4(\alpha -1)^8.
\]

We begin by recalling a lemma from \cite{creutz2024quartic}:
	
\begin{lemma}\cite[Lemma 2]{creutz2024quartic} \label{lemma2} Let $Q$ be a rank 5 quadratic form over a field $k$ of characteristic not equal to 2 and assume that $Q$ is an orthogonal sum of a hyperbolic plane and a rank 3 form. The quadric 3-fold $\mathbb{V}(Q) \subset \PP^4_k$ contains a k-rational line if and only if the conic in $\PP^2_k$ defined by the rank 3 form contains a $k$-rational point.
\end{lemma}

We now apply Lemma \ref{lemma2} to show exactly under which choice of $\nu \in \Omega_k$ the quadric 3-fold $\mathbb{V}(Q_{\infty})$, given by  \eqref{qinfty}, contains a $k_\nu$-rational line. 
	
\begin{lemma} \label{lem: local solubility of Q_infinity}
Let $\nu\in\Omega_{k}$. The quadric $\mathbb{V}(Q_{\infty})$, contains a $k_\nu$-rational line if and only if $\nu\notin\{ t,\infty\}$.
\end{lemma}

\begin{proof}
Observe $Q_\infty$ is the sum of a hyperbolic plane and the rank $3$ quadric
\begin{equation*}
	\frac{\alpha }{t}u_2^2+t^2u_3^2+2\alpha u_2u_4+\frac{3\alpha t-t}{\alpha+1 }u_4^2.
\end{equation*}
Diagonalising over $k_\nu$, we may write this quadric (up to the leading constant) as
\begin{equation*}
	(u_2+tu_4)^2+\frac{t^3}{\alpha} u_3^2-\frac{(\alpha-1)^2}{\alpha(\alpha +1)}t^2u_4^2.
\end{equation*}
		
Hence the quadric corresponds to the quaternion algebra $(\alpha (\alpha +1),t(\alpha +1))$. By our choice of $\alpha$, we have
\[
    (\alpha (\alpha +1),t(\alpha +1))=(\alpha,t)\in \Br(k),
\]
which is ramified exactly at $\nu=t,\infty$. The result now follows from Lemma \ref{lemma2}.
\end{proof}

We now study the local solubility at the prime $\nu=t$, which is essential for the obstruction.

\begin{lemma} \label{lemma4} 
For all $b \in \oo_{t}$, the quadric $\mathbb{V}(bQ_0 + Q_{\infty}) \subset \Proj^4_{k_t}$ contains no $k_t$-rational lines. 
\end{lemma}

\begin{proof}
First note that all roots over $\Bar{k_t}$ of $f(x)$ have positive $t$-valuation and thus $1/b$ is not a root of $f(x)$, so $bQ_0+Q_{\infty}$ has rank 5. 
		
The quadratic form $bQ_0+Q_{\infty}$ is the orthogonal sum of the rank 2 form
\[
	d\left(b\,u_0^2 - 2\,u_0u_1 - \,tb\,u_1^2\right),
\]
and a rank 3 form which diagonalises as		\begin{equation}\label{eqn:rank3}   \beta_2u_2^2+\beta_3u_3^2+\beta_4u_4^2,
\end{equation}
where
\begin{align*}
	&\beta_2 = t^5b + t^4b + \alpha t, \\
	&\beta_3 = \frac{t^5b(t+1) - t^3\alpha^2b^2 + t^2\alpha}{t^3b(t+1) + \alpha}, \\
	&\beta_4 = \frac{-t^5(t+1)b^3(\alpha+1)+ t^4(t+1)b(3\alpha-1)- t^2b^2\,\alpha(\alpha-1)^2- t\,\alpha(\alpha-1)^2}{(\alpha+1)\bigl(t^3b(t+1) - t\alpha^2b^2 + \alpha\bigr)}.
\end{align*}
		
Since $b\in \oo_t$, the discriminant of the rank 2 form is 
\[
	-d^2(tb^2+1)\equiv -1 \mod k_t^{\times 2}.
\]
Thus the rank 2 form is hyperbolic over $k_t$. We apply Lemma \ref{lemma2} to $bQ_0+Q_{\infty}$ and deduce that $\mathbb{V}(bQ_0+Q_{\infty})$ contains a line over $k_t$ if and only if ~\eqref{eqn:rank3} has a $k_t$-point. Since $b\in \oo_t$, the Hilbert symbol associated to this conic is 
\[
(-\alpha t, t(\alpha+1))_t=(-\alpha t,\alpha(\alpha+1))_t,
\] 
which for our choice of $\alpha$ is non-trivial.
\end{proof}

\subsection{Brauer--Manin obstruction on $\mathcal{G}_d$}

In this section we prove $X_d$ admits no quadratic points under certain assumptions on $d$. \par
Let $\mathcal{G}_d$ be the fourfold $\pi: \mathcal{G}_d\rightarrow \PP^1_k$ parametrising lines on quadrics in the pencil associated to $X_d$, as in Section \ref{section:thevarietyG}. Since both $\epsilon_2$ and $\epsilon_3$ are nonsquares in their respective \'etale algebras, and both have square norms, by \cite[Proposition 5.1]{CreutzViray2023} we have that $\Br(\mathcal{G}_d)/\Br_0(\mathcal{G}_d)\cong \Z/2\Z$ is generated by 
\begin{equation}\label{eq: first rep for brauer group}
\mathcal{A}:=\pi^*(d\alpha t(t+1),x^2+t),  
\end{equation}
which can be also written in the form 
\begin{equation}\label{eq: second rep for brauer group}
\mathcal{A}=\pi^*\operatorname{Cor}_{k(\theta)/k}(\epsilon_3,x-\theta)+[\left(\mathcal{G}_d\right)_\infty]\in \Br(\mathcal{G}_d),
\end{equation}  
where by Lemma~\ref{lem: local solubility of Q_infinity} we have
\begin{equation*}
[\left(\mathcal{G}_d\right)_\infty]=(\alpha,t).
\end{equation*}

By \cite[Theorems 1.2 and 1.8]{reid1972complete}, the primes of bad reduction of $\mathcal{G}_d$ are the primes dividing $\Delta$, so let
\[
	S_d:=\{\nu \mid d \xi \} \cup \{t,t+1,\infty\},
\]
denote the set of primes of bad reduction of $\mathcal{G}_d$.\par
We first prove the associated local invariant maps are zero at the primes of good reduction:

\begin{lemma} \label{constant zero evaluation}
For any $d\in k^{\times}$ and for any prime $\nu \notin S_d$, we have $\inv_\nu\mathcal{A}=0$.   
\end{lemma}

\begin{proof}
Since $\mathcal{G}_d$ has good reduction at $\nu$ , by Proposition~\ref{proposition: good reduction at a place imples constant inv map}  we have that $\inv_\nu\mathcal{A}$ is constant. Thus it suffices to prove $\inv_\nu\mathcal{A}(Q)=0$ when $Q$ is a $k_\nu$-rational line on $Q_\infty$ (which exists by Lemma~\ref{lem: local solubility of Q_infinity}). We first rewrite 
\begin{align*}
\mathcal{A}&=\pi^*(d\alpha t(t+1), x^2+t)\\
&=\pi^*(d\alpha t(t+1), (x^2+ty^2)/y^2)\\
&= \pi^*(d\alpha t(t+1), (x^2+ty^2)/x^2) \in \pi^*(\Br(k(\Proj^1))),
\end{align*}
so that
\begin{align*}
\inv_\nu\mathcal{A}(Q)&=\inv_\nu\pi^*(d\alpha t(t+1), (x^2+ty^2)/x^2)\lvert_{y=0}\\
&=\inv_\nu\pi^*(d\alpha t(t+1), 1)\\
&=0.\qedhere
\end{align*}
\end{proof}

We now exhibit a Brauer--Manin obstruction on $\mathcal{G}_d$ under some assumptions on $d$:
\begin{lemma}\label{prop: under hypothesis BMO to quad point}
Let $D \in \F_p[t]$ such that the following hold:
\begin{enumerate}
	\item \label{cond: ep3 is a square for all v/d} For all primes $\nu\mid D$ with $\nu \neq t$,  $\epsilon_3$ is a square in $k_\nu \otimes k(\theta)$;
	\item \label{cond: d or dy is a square at ky}$\alpha D$ is a square in $k_t$;
	\item \label{cond: dx square at infinity}$D$ has odd degree and leading coefficient equal to $-1$;
	\item \label{cond: get rid of primes dividing delta}For all primes $\nu\mid \xi(t+1)$, either $D$ is a square in $k_\nu$ or $\epsilon_3$ is a square in $k_\nu\otimes k(\theta)$.  
\end{enumerate}
For such $D$, let $d=D/\alpha t(t+1)$. Then $\mathcal{G}_d(k)=\varnothing$. 
\end{lemma}

\begin{proof}
It suffices to prove $\mathcal{G}_d$ has a Brauer--Manin obstruction to the Hasse principle. \par
Recall that by \eqref{eq: second rep for brauer group} and Lemma \ref{lem: local solubility of Q_infinity}, the Brauer element $\mathcal{A}=\pi^*(D,x^2+t) $ can be represented as \[  \mathcal{A}=\pi^*\operatorname{Cor}_{k(\theta)/k}(\epsilon_3,x-\theta)+ (\alpha, t).\]
By Lemma~\ref{constant zero evaluation} for all $\nu\notin S_d$, we have $\inv_\nu\mathcal{A}=0$. 
        
For $\nu=\infty$, let $Q\in \mathcal{G}_d(k_\infty)$ and let $\pi(Q)=x_0 \in \PP^1_{k_\infty}$. By Lemma~\ref{lem: local solubility of Q_infinity}, we have that $x_0 \neq \infty \in \Proj^1_{k_\infty}$, and hence we may view $x_0\in k_\infty$. If $v_\infty(x_0)<0$, then $x_0^2+t $ is a square in $k_{\infty}$ and hence we have
\begin{equation*}
\inv_{\infty}Q^*\mathcal{A}=\left(D,x_0^2+t\right)_\infty=0.
\end{equation*}
         
If $v_\infty(x_0) \geq 0$, then $x_0^2+t = t\in k_{\infty}^{\times}/(k_{\infty}^{\times})^2$. Since $D$ is odd degree and has leading coefficient equal to $-1$ we have that $D= -t\in k_{\infty}^{\times}/(k_{\infty}^{\times})^2$. In particular we have 
\begin{equation*}
\inv_{\infty}Q^*\mathcal{A}=\left(D,x_0^2+t\right)_\infty=\left(-t,t\right)_\infty=0.
\end{equation*}
Therefore $\inv_\infty Q^*\mathcal{A}=0$ for all $Q\in \mathcal{G}_d(k_\infty)$.\par
For $\nu\mid D$ and $\nu \neq t$, by \eqref{cond: ep3 is a square for all v/d} the second representation of  $\mathcal{A}$ yields $\inv_\nu\mathcal{A}=\inv_\nu(\alpha,t)=0$. 

For $\nu\mid \xi(t+1)$, by \eqref{cond: get rid of primes dividing delta} either $ D$ is a square in $k_\nu$ or $\epsilon_3$ is a square in $k_\nu\otimes k(\theta)$. If the former holds, then we have $\inv_\nu\mathcal{A}=\inv_{\nu}(1,x^2+t)=0$. If the latter holds, then we have $\inv_\nu\mathcal{A}=\inv_\nu(\alpha,t)=0$. 
\par
For $\nu=t$, by Lemma~\ref{lemma4}, if $Q\in\mathcal{G}_d(k_t)$, then $\pi(Q)=[ta:1]$, for $a\in \OO_{t}$. In particular, since \eqref{cond: d or dy is a square at ky} gives that $\alpha D$ is a square in $k_t$ we have
\begin{equation*} \inv_t\mathcal{A}=\inv_t\left(D,t^2a^2+t\right)=\inv_t\left(\alpha,t\right)+\inv_t\left(\alpha,ta^2+1\right)=\frac{1}{2}.
\end{equation*}
Therefore for any $(Q_\nu)_{\nu}\in\prod_{\nu\in\Omega_k}\mathcal{G}_d(k_\nu)$ we have
\begin{equation*}
\sum_{\nu\in\Omega_k}\inv_\nu\mathcal{A}(Q_\nu)=\frac{1}{2},
\end{equation*}
and hence there is a Brauer--Manin obstruction to the Hasse principle on $\mathcal{G}_d$.
\end{proof}
	
\subsection{Existence of $X_d$ for $p\notin\{3,7,11\}$} \label{sec: existence of Xd for p not 3,7,11}
    
In this section we will prove Theorem~\ref{thm:maindp4} in the case $p\not\in\{3,7,11\}$. We will do this by utilising Chebotarev's density theorem over global function fields. Unless explicitly stated otherwise, the results in this section hold for all $p>3$.
    
Let $K=\F_q(C)$ be a global function field, for some curve $C/\F_q$, and let $L/K$ be a finite Galois extension. For a prime $P$ in $K$, we define the degree of the prime to be the degree of its residue field. We denote by $(P,L/K)$ the Artin conjugacy class of $P$ (see \cite[Page~122]{rosen} for precise definition). The extension $L/K$ is \textbf{geometric} if $\F_q$ is algebraically closed in $L$.
    
\begin{theorem}\cite[Theorem 9.13B]{rosen}\label{theorem:CDTforGFFs}
Let $L/K$ be a geometric Galois extension of global function fields with $G=\Gal(L/K)$. Let $C\subset G$ be a non-empty conjugacy class. For every sufficiently large integer $N$, there exists a prime $P$ of degree $N$ with $(P,L/K)=C$. 
\end{theorem} 	  

We will use the quadratic reciprocity law for function fields. First, we recall the Legendre symbol in the classical setting: for $c\in \mathbb F_p^{\times},$ we define
\begin{equation*}
\left(\frac{c}{p}\right)=\begin{cases}
			1, &\text{if }c\text{ is a square in }\mathbb F_p,\\
			-1, &\text{ otherwise}.
		\end{cases}   
\end{equation*}
Now let $a\in \mathbb F_p[t]$ and let $\nu\in \Omega_k\backslash\{\infty\}$. Suppose $a\in \oo^{\times}_{\nu} $. We define the Legendre symbol 
\begin{equation*}
\bigg(\frac{a}{\nu}\bigg)=\begin{cases}
			1, &\text{if }a\text{ is a square modulo }\nu,\\
			-1, &\text{ otherwise}.
		\end{cases}   
\end{equation*}
If the prime ideal $\nu$ is generated by an irreducible polynomial $b\in \mathbb F_p[t],$ we write $\left( \frac{a}{b} \right)$ to mean $\left( \frac{a}{\nu} \right)$. 
\begin{theorem}\cite[Theorem~3.5]{rosen}\label{theorem:reciprocity}
Let $a,b\in \F_p[t]$ be two distinct non-zero irreducible polynomials. Then
\[
	\bigg( \frac{a}{b} \bigg)= \left(\frac{-1}{p}\right)^{\deg(a)\deg(b)}\left(\frac{\operatorname{lead}(a)}{p}\right)^{\deg(b)}\left(\frac{\operatorname{lead}(b)}{p}\right)^{\deg(a)}\bigg( \frac{b}{a} \bigg),
\]
where $\operatorname{lead}(a),\operatorname{lead}(b)$ denote the leading coefficients of $a,b$ respectively.
\end{theorem} 
	
\begin{lemma}\label{lem: quad rec trick for changing reps}
Let $\nu\in \Omega_k\setminus\{\infty\}$ be a prime and let $\delta\in\{-1,1\}$. There exists a polynomial $q_{\nu,\delta}\in \mathbb F_p[t]$ generating $\nu$ as an ideal, such that for any $\nu'\in \Omega_k\setminus\{\infty,\nu\}$ generated by an odd degree polynomial $D$ with $\operatorname{lead}D=-1$, we have that $\nu'$ splits in $k(\sqrt{q_{\nu,\delta}})$ if and only if
\begin{equation*}
	\left(\frac{D}{\nu}\right)=\delta.
\end{equation*}
\end{lemma}

\begin{proof}
Let $q$ be a monic polynomial in $\mathbb F_p[t]$ generating the ideal $\nu$. By Theorem~\ref{theorem:reciprocity} we have
\begin{equation*}
	\left(\frac{D}{q}\right)=\left(\frac{-1}{p}\right)^{\deg q}  \left(\frac{\operatorname{lead}(D)}{p}\right)^{\deg q} \left(\frac{\operatorname{lead}(q)}{p}\right)^{\deg D} \bigg(\frac{q}{D}\bigg) = \bigg(\frac{q}{D}\bigg).
\end{equation*}
Similarly, we have
\begin{equation*}
	\left(\frac{D}{\alpha q}\right)= \bigg(\frac{\alpha q}{D}\bigg)\left(\frac{\operatorname{lead}(\alpha q)}{p}\right)=-\bigg(\frac{\alpha q}{D}\bigg).
\end{equation*}
Since $q$ and $\alpha q$ generate the same prime $\nu$, we have $\left(\frac{-}{q}\right)=\left(\frac{-}{\alpha q}\right)=\Big(\frac{-}{\nu}\Big)$ and we obtain
\begin{equation*}
	\bigg(\frac{q}{D}\bigg)=   \left(\frac{D}{\nu}\right)= -\bigg(\frac{\alpha q}{D}\bigg).
\end{equation*}
Let 
\begin{equation*}
	q_{\nu,\delta}=\begin{cases}
				q, &\text{if }\delta=1,\\
				\alpha q, &\text{if }\delta = -1.
			\end{cases}   
\end{equation*} 
We have that $\nu'$ splits in $k(\sqrt {q_{\nu,\delta}})$ if and only if $\left(\frac{q_{\nu,\delta}}{D}\right)=1$ (see \cite[Proposition 10.6]{rosen}), and this is equivalent to $\left(\frac{D}{\nu}\right)=\delta$ by our choice of $q_{\nu,\delta}$.  
\end{proof}

\begin{lemma}\label{galois group is S_4}
Let $\epsilon_3$ be as in \eqref{eq: definition of ep2 and ep3}, then we have $\Gal(k(\sqrt{\epsilon_3})^{\Gal}/k)\cong S_4$.
\end{lemma}
	
\begin{proof}
Let $K_1=k(\epsilon_3)^{\Gal}$, and let $K_2=k(\sqrt{\epsilon_3})^{\Gal}$. Let $\sigma \in \Gal(K_2/k)$. Then
\[
\sigma\left(\sqrt{\epsilon_3}\right)^2=\sigma(\epsilon_3)=\sigma_1(\epsilon_3),
\]
for some $\sigma_1\in \Gal(K_1/k)$, so the Galois conjugates of $\sqrt{\epsilon_3}$ are therefore
\[
\pm\sqrt{\epsilon_3},\pm \sqrt{\sigma_1(\epsilon_3)},\pm \sqrt{\sigma_1^2(\epsilon_3)},
\] 
for some $\sigma_1 \in \Gal(K_1/k)$ not fixing $\epsilon_3$.  Since $N_{ k(\epsilon_3)/k}(\epsilon_3)\in k^{\times2}$, the Galois closure of $k(\sqrt{\epsilon_3})/k$ is
\[        K_2=K_1\left(\sqrt{\epsilon_3},\sqrt{\sigma_1(\epsilon_3)}\right).
\]
The Galois group of $K_1/k$ is $S_3$, since $\Delta_{f_3}$ is non-square, so $\Gal(K_2/k)$ is a $V_4$-extension of $S_3$. 

Consider the tower of extensions $ K_1(\sqrt{\epsilon_3})/K_1/k$. Then $\Gal(K_1(\sqrt{\epsilon_3})^{\Gal}/k)=\Gal(K_2/k)$ is a subgroup of the wreath product of $\Gal(K_1(\sqrt{\epsilon_3})/K_1)\cong C_2$ with $\Gal(K_1/k)\cong S_3$ by \cite[Lemma 4.1]{Odoni1985GaloisIterates}. The only such group which is a $V_4$-extension of $S_3$ is $S_4$, see for instance \cite{GroupNames}.
\end{proof}

\begin{lemma}\label{lemma:geometric}
Let  $L:=k(\sqrt{\epsilon_3})^{\Gal}k(\sqrt{t+1})\prod_{\nu\mid \xi}k(\sqrt{q_{\nu,1}})$, then the extension $L/k$ is geometric. Furthermore, if $p\notin\{7,11\}$, then $L$ contains the subfield $k(\sqrt{q_{t,-1}})$.
\end{lemma} 

\begin{proof}  
Since $k(\sqrt{\epsilon_3})^{\Gal}/k$ is an $S_4$-extension by Lemma \ref{galois group is S_4}, it has a unique quadratic subextension with discriminant $\Delta_{f_3}=t\xi \in k^{\times}/k^{\times 2} $. Therefore,
\[
    k(\sqrt{\epsilon_3})^{\Gal} \cap k(\sqrt{t+1})\prod_{\nu\mid \xi}(\sqrt{q_{\nu,1}})=k,
\]
and so we have $\Gal(L)\cong S_4\times C_2^{r+1}$, where $r$ is the number of prime factors of $\xi$. Since every constant extension of $k$ is abelian, and $(S_4\times C_2^{r+1})^{ab}\cong C_2^{r+2}$, a non-trivial constant extension would have to contain the degree 2 constant extension of $k$. Every degree 2 subextension of $L/k$ is geometric. Thus $L/k$ is geometric.
    
By considering $\Delta_{f_3}\prod _{v\mid \xi} q_{v,1}^{-1}\in L$ we see $L$ contains the quadratic extension
\[
k\left(\sqrt{(\alpha+1)(3\alpha-1)^3t}\right),
\]
where $(\alpha+1)(3\alpha-1)^3$ is the leading coefficient of $\xi$. If $p \notin \{7,11\}$, then conditions \eqref{alpha+1 is a square} and \eqref{3alpha-1 is not a square} of Lemma \ref{lem:choiceofalpha} ensure that $k\left(\sqrt{(\alpha+1)(3\alpha-1)^3t}\right)=k\left(\sqrt{\alpha t}\right)$.
\end{proof}
	
\begin{lemma}\label{lem:applyingcheb for not 2,3,7,11}
	Assume $p\notin\{7,11\}$, then there exist infinitely many primes $D\in \Omega_k$ such that $D$ satisfies the conditions of Lemma~\ref{prop: under hypothesis BMO to quad point}.  
\end{lemma}    
\begin{proof}
By Lemma~\ref{lemma:geometric} the field extension $L/k$ is geometric, and hence by applying Theorem~\ref{theorem:CDTforGFFs} to $L/k$ and $C=1$, we obtain infinitely many irreducible polynomials $D\in\F_p[t]$ of odd degree with leading coefficient equal to $-1$ such that the prime ideal $(D)$ splits completely in $L$. Without loss of generality, we may assume $(D)\notin \{\nu\in\Omega_k:\nu\mid \xi t(t+1)\infty\}$. Since $L$ contains the subfield $k(\sqrt{q_{t,-1}})$, by Lemma~\ref{lem: quad rec trick for changing reps} this implies $D\notin k_t^{\times 2}$. In particular condition~\eqref{cond: d or dy is a square at ky} of Lemma~\ref{prop: under hypothesis BMO to quad point} holds. Similarly by definition of $L$ we have the following:
\begin{enumerate}[label=(\roman*)]
	\item \label{cond: ep3 sqaure}$\epsilon_3\in\left(k_D\otimes k(\theta)\right)^{\times 2}$;
	\item \label{cond: primes dividing xi and t+1 have D a square in knu}For all $\nu\mid (t+1)\xi$, we have that $D\in k_\nu^{\times 2}$.
\end{enumerate}
In particular conditions \eqref{cond: ep3 is a square for all v/d} and \eqref{cond: get rid of primes dividing delta} of Lemma~\ref{prop: under hypothesis BMO to quad point} follow immediately from conditions \ref{cond: ep3 sqaure} and \ref{cond: primes dividing xi and t+1 have D a square in knu}, respectively. Since $\operatorname{deg}D$ is odd and $\operatorname{lead}D=-1$, condition~\eqref{cond: dx square at infinity} also follows.
\end{proof}
	
\subsection{Existence of $X_d$ for $p\in\{7,11\}$}
In this section we will prove Theorem~\ref{thm:maindp4} in the case $p\in\{7,11\}$. The method will be similar to Section~\ref{sec: existence of Xd for p not 3,7,11}.
	
\begin{lemma}\label{lem:711}
Assume $p\in\{7,11\}$, then for all $\nu\mid \xi$ we have that $\epsilon_3\in \left(k_\nu\otimes k(\theta)\right)^{\times 2}$.   
\end{lemma}

\begin{proof}
The following calculations were verified in Magma.
Consider first the case $p=7$ and $\alpha=3$. Then $\xi$ factors as $\xi=\omega_1\omega_2$ in $\F_7(t)$, where 
\begin{align*}
\omega_1:&=t^4+5t^2+2t+4,\\
\omega_2:&=t^{10}+4t^9+t^8+3t^7+5t^6+2t^5+5t^3+t^2+4t+6.
\end{align*}

Let $k_{\omega_1}$ and $k_{\omega_2}$ denote the completions of $k$ at $\omega_1$ and $\omega_2$ respectively. We want to show $\epsilon_3\in \left(k_{\omega_1}\otimes k(\theta)\right)^{\times 2}$ and $\epsilon_3\in \left(k_{\omega_2} \otimes k(\theta)\right)^{\times 2}$.

By factoring $f_3$ over $k_{\omega_1}$, we have that $f_3=f_{3,1}f_{3,2}$, where $f_{3,1},f_{3,2}\in k_{\omega_1}[x]$ are monic irreducible polynomials of degree 1 and 2, respectively. Let $(\epsilon_{3,1},\epsilon_{3,2})$ be the image of $\epsilon_3$ under the isomorphism
\begin{align*}
	k_{\omega_1}\otimes_k k[x]/f_3(x)&\cong k_{\omega_1}[x]/f_3(x)\\ 
	&\cong k_{\omega_1}[x]/f_{3,1}(x)\oplus k_{\omega_1}[x]/f_{3,2}(x)\\
	&\cong k_{\omega_1}\oplus k_{\omega_1}(\theta_2).
\end{align*}

Let $\pi_1$ be a uniformiser of $k_{\omega_1}$, and let $a$ be a lift to $k_{\omega_1}$ of a generator of its residue field. Then
\[
\epsilon_{3,1}\equiv 4a^3 + 4a^2 + 6
\in \left(\OO_{k_{\omega_1}}/\pi_1\right)^{\times2}
\]
is a square. 
        
Since $v_{\pi_1}(\theta_2)=0$, $\theta_2$ is not a uniformiser of $k_{\omega_1}(\theta_2)$. Let $f_{3,2}=x^2+bx+c$. We note that $v_{\pi_1}\left(\Delta_{f_{3,2}}\right)=1$, so $v_{\pi_1}\left(\theta_2+\frac{b}{2}\right)=\frac{1}{2}$. Therefore, the element 
\[
\pi_2:=\theta_2+\frac{b}{2},
\]
is a uniformiser of $k_{\omega_1}(\theta_2)$. Then
\[   
\epsilon_{3,2}\equiv 4a^3+a^2+6\in \left(\oo_{k,\omega_1}/\pi_2\right)^{\times2}
\]
and again applying Hensel's Lemma we conclude $\epsilon_3$ is a square in $k_{\omega_1}(\theta_2)$. 

We repeat a similar argument with $k_{\omega_2}$. Again, $f_3=f_{3,1}f_{3,2}$, where $f_{3,1},f_{3,2}\in k_{\omega_1}[x]$ are monic irreducible polynomials of degree 1 and 2, respectively. Then in the corresponding \'etale algebras, both
\[
\epsilon_{3,1}\equiv 4a^9 + a^8 + 4a^7+2a^2\in \left(\oo_{k,\omega_2}/\pi_1\right)^{\times2}
\]
and
\[   
\epsilon_{3,2}\equiv 4a^9+a^8+4a^7+6a^2\in \left(\oo_{k,\omega_2}/\pi_2\right)^{\times2}
\]
are square.

Now we consider the case $p=11$ and $\alpha=8$. Then $\xi$ factors as $\xi=\omega_1 \omega_2 \omega_3$, where
\begin{align*}
\omega_1:&=t^2+2t+10,\\
\omega_2:&=t^6+4t^5+6t^4+9t^3+2t^2+6t+7,\\
\omega_3:&=t^6+9t^5+5t^4+5t^3+3t^2+2t+10.
\end{align*}

Over $k_{\omega_1}$, $f_3$ factors into a linear and a quadratic factor, $f_3=f_{3,1}f_{3,2}$. As before, in the corresponding \'etale algebras, both
\[
\epsilon_{3,1}\equiv 9a+8 \in \left(\oo_{k,\omega_1}/\pi_1\right)^{\times2}
\]
and
\[
\epsilon_{3,2}\equiv 7a+1\in \left(\oo_{k,\omega_1}/\pi_2\right)^{\times2}
\]
are square.

Over $k_{\omega_2}$, $f_3$ factors into a linear and a quadratic factor, $f_3=f_{3,1}f_{3,2}$. In the corresponding \'etale algebras, both
\[
\epsilon_{3,1}\equiv a^5+3a^4+3a^3+8a^2+8 \in \left(\oo_{k,\omega_2}/\pi_1\right)^{\times2}
\]
and
\[
\epsilon_{3,2}\equiv 3a^5+9a^4+9a^3+a^2+2 \in \left(\oo_{k,\omega_2}/\pi_2\right)^{\times2}
\]
are square. 

Over $k_{\omega_3}$, $f_3$ factors into a linear and a quadratic factor, $f_3=f_{3,1}f_{3,2}$. In the corresponding \'etale algebras, both
\[
\epsilon_{3,1}\equiv 7a^4+8a^2+5a+5 \in \left(\oo_{k,\omega_3}/\pi_1\right)^{\times2}
\]
and
\[
\epsilon_{3,2}\equiv 10a^4+a^2+4a+4 \in \left(\oo_{k,\omega_3}/\pi_2\right)^{\times2}
\]
are square. 
\end{proof}
	
\begin{lemma}\label{lem: p =3,11 cheb}
Assume $p\in\{7,11\}$, then there exist infinitely many primes $D\in \Omega_k$ such that $D$ satisfies the conditions of Lemma~\ref{prop: under hypothesis BMO to quad point}. \end{lemma}

\begin{proof}
Consider the field extension $L':=k(\sqrt{\epsilon_3})^{\Gal}(\sqrt{q_{t,-1}},\sqrt{q_{t+1,1}})$. By Lemma~\ref{lemma:geometric}, the extension $k(\sqrt{\epsilon_3})^{\Gal}/k$ is geometric. By Lemma~\ref{galois group is S_4}, the Galois group of $k(\sqrt{\epsilon_3})^{\Gal}/k$ is $S_4$, and so the extension has a unique quadratic subextension with discriminant $\Delta_{f_3}=t\xi$. The maximal abelian subextension of $L'/k$ is $k(\sqrt{q_{t,-1}},\sqrt{q_{t+1,1}},\sqrt{t\xi})/k$ which is geometric, and hence $L'/k$ must be geometric. Thus by applying Theorem~\ref{theorem:CDTforGFFs} to the extension $L'/k$ and $C=1$, we obtain infinitely many irreducible polynomials $D\in\F_p[t]$ of odd degree with leading coefficient equal to $-1$ and such that the following hold:
\begin{enumerate}[label=(\roman*)]
	\item \label{cond i for p is 7,11} $\epsilon_3\in\left(k_D\otimes k(\theta)\right)^{\times 2}$;
    \item  \label{cond ii for p is 7,11} $D\notin k_t^{\times 2}$.
    \item \label{cond: 2 for p is 7,11} $D\in k_{t+1}^{\times 2}$;
\end{enumerate}

Since $D$ has leading coefficient -1 and has odd degree, it satisfies condition \eqref{cond: dx square at infinity}. 
Conditions \eqref{cond: ep3 is a square for all v/d}, \eqref{cond: d or dy is a square at ky} are satisfied by \ref{cond i for p is 7,11}, \ref{cond ii for p is 7,11}, respectively. Condition ~\eqref{cond: get rid of primes dividing delta} is satisfied by \ref{cond: 2 for p is 7,11} and Lemma ~\ref{lem:711}. 
\end{proof}

\section{Existence for \texorpdfstring{$p=3$}{}}\label{sec:p=3}
	
Let $k=\F_3(t)$, and let $X_d:=\{Q_0=Q_\infty=0\}\subset \PP^4_k$ be given by 
\begin{align}
	Q_0:&=d(u_0^2-tu_1^2)+(t+t^2)u_2^2+tu_2u_3-tu_3u_4,\\
	Q_\infty:&=du_0u_1-u_2^2+ tu_3^2-u_2u_4+u_4^2.
\end{align}

The characteristic polynomial of $X_d$ is
\[
    f(x):=2td^2(x^2+t)f_3(x),
\]
where 
\[
    f_3(x):=x^3+(t^2+t)x^2+tx+2t^3+2t^2.
\]
Let $\theta_3\in k[x]/f_3(x)$ denote the image of $x$ under the canonical map $k[x]\rightarrow k[x]/f_3(x)$, so that
\[
	k[x]/f(x)\cong k(\sqrt{-t})\oplus k(\theta_3).
\]
The $\epsilon_i$ are
\[
\epsilon_2= dt(t-1)\in k(\sqrt{-t})^{\times}/k(\sqrt{-t})^{\times2},
\]
and
\[
    \epsilon_3= (t^7+2t^6+2t^3+t^2+1)\theta_3\in k(\theta_3)^{\times}/k(\theta_3)^{\times2}.
\]
Both $\epsilon_2$ and $\epsilon_3$ are not squares and both $\epsilon_2$ and $\epsilon_3$ have square norms.
The discriminant of $f$ is
\[
	\Delta=2d^{16}t^{20}(t+1)^4(t^2+2t+2)(t^4+2t^3+t+1),
\]
and the discriminant of $f_3$ is
\[
    \Delta_{f_3}=t^3(t^2+2t+2)(t^4+2t^3+t+1).
\]
	
By \cite[Theorems 1.2 and 1.8]{reid1972complete}, the primes of bad reduction of $\mathcal{G}_d$ are the primes dividing $\Delta$, so let
\[
	S:=\{ \nu \mid d\} \cup \{t,t+1,t^2+2t+2,t+4+2t^3+t+1,\infty\}
\]
be the set of primes of bad reduction of $\mathcal{G}_d$. For all $\nu \in S$ and $P\in X(k_{\nu})$ we compute $\inv_{\nu} \mathcal{A}(P_{\nu})$.
	
\begin{lemma}\label{lem: local solubility of Q_infinity for 3}
Let $\nu\in\Omega_{k}$. The quadric $\mathbb{V}(Q_{\infty})$ contains a $k_\nu$-rational line if and only if $\nu\notin\{ t,\infty\}$.
\end{lemma}       
	
\begin{proof}
Observe $Q_\infty$ is the sum of a hyperbolic plane and the rank $3$ quadric
\begin{equation*}
-u_2^2+ tu_3^2-u_2u_4+u_4^2.
\end{equation*}
Diagonalising over $k_\nu$, we may write this quadric (up to the leading constant) as
\begin{equation*}
2(u_2+2u_4)^2+tu_3^2+2u_4^2.
\end{equation*}

Hence the quadric corresponds to the quaternion algebra $(2,t)\in \Br k_\nu$, which is ramified at exactly $\nu=t,\infty$. The result now follows from Lemma \ref{lemma2}.
\end{proof}
	
\begin{lemma} \label{lemma4 for 3} For all $b \in \oo_{t}$, the quadric $\mathbb{V}(bQ_0 + Q_{\infty}) \subset \PP^4_{k_t}$ contains no $k_t$-rational lines. 
\end{lemma}       

\begin{proof}
First note that all roots over $\Bar{k_t}$ of $f(x)$ have positive $t$-valuation and thus $1/b$ is not a root of $f(x)$, so $bQ_0+Q_{\infty}$ has rank $5$. 
		
The quadratic form $bQ_0+Q_{\infty}$ is the orthogonal sum of the rank 2 form
\[
d\left(b\,u_0^2 - 2\,u_0u_1 - \,tb\,u_1^2\right),
\]
and a rank 3 form which diagonalises as
		
\begin{equation}\label{eqn:rank3conic}
\beta_2u_2^2+\beta_3u_3^2+\beta_4u_4^2,
\end{equation}
where
\begin{align*}
	\beta_2 &= b(t^2+t)+2,\\
	\beta_3 &= \frac{b\big(b(2t^2)+t^3+t^2\big)+2t}{b(t^2+t)+2},\\
	\beta_4 &= \frac{b\big(b^2(2t^3+2t^2)+b t+(t^2+t)\big)+1}{b\big(b(2t)+t^2+t\big)+2}.
\end{align*}
		
Since $b\in \oo_t$, the discriminant of the rank 2 form is
\[
-d^2(tb^2+1)\equiv -1 \mod (k_t^{\times})^2
\]
so the rank 2 form is isotropic over $k_t$. We then apply Lemma \ref{lemma2} to $bQ_0+Q_{\infty}$ and deduce that $\mathbb{V}(bQ_0+Q_{\infty})$ contains a line over $k_t$ if and only if ~\eqref{eqn:rank3conic} has a $k_t$-point. Since $b\in \oo_t$, the Hilbert symbol associated to this conic is $(2,t)_t=-1$.
\end{proof}
	
Let $\xi$ denote the squarefree part of $\Delta$, so
\[
	\xi:=2\omega_1 \omega_2,
\]
where 
\[
\omega_1=t^2+2t+2,\qquad \omega_2=t^4+2t^3+t+1.
\]
	
Let $k_{\omega_1}$ and $k_{\omega_2}$ denote the completions of $k$ at $\omega_1$ and $\omega_2$, respectively. 
	
\begin{lemma}\label{proposition:IIIepsilon3squareink_p}
Let $\omega_1,\omega_2$ and $\epsilon_3$ be as above, then $\epsilon_3$ is a square in $k_{\omega_1}\otimes_k k(\theta)$ and $k_{\omega_2}\otimes_k k(\theta)$. 
\end{lemma}          
	
\begin{proof}
The calculations are identical in shape to the ones in Lemma \ref{lem:711} and have been verified in Magma. 
Over $k_{\omega_1}$, $f_3$ factors into a linear and a quadratic factor, $f_3=f_{3,1}f_{3,2}$.  Let $(\epsilon_{3,1},\epsilon_{3,2})$ be the image of $\epsilon_3$ under the isomorphism $k_{\omega_1}\otimes_k k[x]/f_3(x) \cong k_{\omega_1}\oplus k_{\omega_1}(\theta_2)$. Let $\pi_1$ be a uniformiser for $k_{\omega_1}$, let $a$ be a lift to $k_{\omega_1}$ of a generator of its residue field, and let $\pi_2$ be a uniformiser for $k_{\omega_1}(\theta_2)$. Then
\[
\epsilon_{3,1}\equiv 2\in \left(\oo_{k,\omega_1}/\pi_1\right)^{\times2},\qquad \epsilon_{3,2}\equiv a^2\in \left(\oo_{k,\omega_1}/\pi_2\right)^{\times2}
\]
are both  square.

The same argument holds for $k_{\omega_2}$. Over $k_{\omega_2}$, $f_3$ factors into a linear and a quadratic factor, and in the corresponding \'etale algebras,  
\[
\epsilon_{3,1}\equiv a^3+2a+2\in \left(\oo_{k,\omega_2}/\pi_1\right)^{\times2},\qquad \epsilon_{3,2}\equiv a^3+2a^2+1\in \left(\oo_{k,\omega_2}/\pi_2\right)^{\times2}
\]
are both square.
\end{proof}
	
We combine all of the above into the following Lemma~\ref{lem:goodd} and compute the Brauer--Manin obstruction on $\mathcal{G}_d$.
	
\begin{lemma}\label{lem:goodd}
Assume $d\in \F_3[t]$ satisfies the following:
\begin{enumerate}
	\item \label{cond: d is a square infinity for p=3} $d\in k_{\infty}^{\times2}$;
	\item \label{cond: dt is a square t+1 for p=3}$dt\in k_{t+1}^{\times2}$;
	\item \label{cond: dt is a square t for p=3} $dt\in k_t^{\times2}$;
	\item \label{cond: e_3 is a square at d for p=3} for all $\nu \mid d$ and $\nu \neq t$, $\epsilon_3\in (k_\nu\otimes_k k(\theta))^{\times2}$.
\end{enumerate}
Then $\mathcal{G}_d(k)=\varnothing$.	\end{lemma}       	
	
\begin{proof}
It suffices to prove $\mathcal{G}_d$ has a Brauer--Manin obstruction to the Hasse principle. The Brauer group $\Br(\mathcal{G}_d)/\Br_0(\mathcal{G}_d)\cong \Z/2\Z$ is generated by 
\begin{equation}\label{eq: first rep for brauer group for p=3}
\mathcal{A}:=\pi^*(\epsilon_2,x^2+t),  
\end{equation}
which can be also written in the form 
\begin{equation}\label{eq: second rep for brauer group for p=3}
\mathcal{A}=\pi^*\operatorname{Cor}_{k(\theta)/k}(\epsilon_3,x-\theta)+[\left(\mathcal{G}_d\right)_\infty]\in \Br(\mathcal{G}_d),
\end{equation}  
where by Lemma~\ref{lem: local solubility of Q_infinity for 3} we have
\begin{equation*}
[\left(\mathcal{G}_d\right)_\infty]=(2,t).
\end{equation*}
		
By Proposition~\ref{proposition: good reduction at a place imples constant inv map}, $\inv_\nu\mathcal{A}$ is constant for all $\nu\notin S$. Let $Q$ be a $k_\nu$-rational line on $Q_\infty$ (which exists by Lemma~\ref{lem: local solubility of Q_infinity for 3}). In particularity $\pi(Q)=[1:0]$. Changing the representation of $\mathcal{A}$ we have
\begin{align*}
\mathcal{A}=\pi^*(\epsilon_2, x^2+t)=\pi^*(\epsilon_2, (x^2+ty^2)/y^2) = \pi^*(\epsilon_2, (x^2+ty^2)/x^2) \in \pi^*(\Br(k(\PP^1))),
\end{align*}
and hence
\begin{equation*}
\inv_\nu\mathcal{A}(Q)=\inv_\nu Q^*\pi^*(\epsilon_2, (x^2+ty^2)/x^2)=\inv_\nu\pi^*(\epsilon_2, 1)=0.
\end{equation*}
Now consider $\nu \in S$ and recall that $\epsilon_2=dt(t-1)$.
		
For $\nu=\infty$, since $d\in k_\infty^{\times2}$ and $t(t-1)\in k_{\infty}^{\times2}$, we have $\inv_\infty\mathcal{A}= 0$. 
		
For $\nu=t+1$, since $dt\in k_{t+1}^{\times2}$ and $(t-1)\in k_{t+1}^{\times2}$, we have $\inv_{t+1}\mathcal{A}= 0$. 
		
For $\nu=t$, by Lemma~\ref{lemma4 for 3}, if $Q\in\mathcal{G}_d(k_t)$ then $\pi(Q)=[ta:1]$, for $a\in \OO_{t}$. In particular, since $dt\in k_t^{\times2}$, we have
\begin{equation*} \inv_t\mathcal{A}(Q)=\inv_t\pi^*\left(t-1,t^2a^2+t\right)=\inv_t\pi^*\left(t-1,t\right)+\inv_t\pi^*\left(t-1,ta^2+1\right)=\frac{1}{2}.
\end{equation*}
	
For $\nu\mid d$ and $\nu \neq t$, condition (4) and ~\eqref{eq: second rep for brauer group for p=3} give $\inv_\nu\mathcal{A}=\inv_\nu(2,t)= 0$. 
		
For $\nu\mid \xi$,   $\epsilon_3$ is a square in $k_\nu\otimes k(\epsilon_3)$ by Lemma~\ref{proposition:IIIepsilon3squareink_p}. By ~\eqref{eq: second rep for brauer group for p=3} again we have $\inv_\nu\mathcal{A}=\inv_\nu(2,t)= 0$. 
		
Therefore for any $(Q_\nu)_{\nu}\in\prod_{\nu\in\Omega_k}\mathcal{G}_d(k_\nu)$ we have
\begin{equation*}
\sum_{\nu\in\Omega_k}\inv_\nu\mathcal{A}(Q_\nu)=\frac{1}{2},
\end{equation*}
and hence there is a Brauer--Manin obstruction to the Hasse principle on $\mathcal{G}_d$.
\end{proof}

\subsection{Existence of $X_d$ for $p=3$}
	
\begin{lemma}\label{lem:applyingcheb}
There exist infinitely many prime elements $D\in \F_3[t]$ such that $d=Dt$ satisfies the conditions of Lemma \ref{lem:goodd}. 
\end{lemma}

\begin{proof}
Since $k(\sqrt{\epsilon_3})^{\Gal}\cap k(\sqrt{t},\sqrt{t+1})=k$, we have that
\begin{align*}
&\Gal\left(k\left(\sqrt{t},\sqrt{t+1},\sqrt{\epsilon_3}\right)^{\Gal}/k\right)\\
\cong &\Gal\left( k\left(\sqrt{t},\sqrt{t+1}\right)/k\right)\times \Gal\left(k(\sqrt{\epsilon_3})^{\Gal}/k \right) \\
\cong &\Gal\left( k\left(\sqrt{t}\right)/k\right)\times \Gal\left( k\left(\sqrt{t+1}\right)/k\right)  \times \Gal\left(k(\sqrt{\epsilon_3})^{\Gal}/k \right).
\end{align*}

Consider $\tau=(\sigma_1,\sigma_2,1)$ in $\Gal\left(k\left(\sqrt{t},\sqrt{t+1},\sqrt{\epsilon_3}\right)^{\Gal}/k\right)$, where $\sigma_1$ generates the Galois group of $k(\sqrt{t})/k$ and $\sigma_2$ generates the Galois group of $k(\sqrt{t+1})/k$. One can verify in Magma the extension $k\left(\sqrt{\epsilon_3}\right)^{\Gal} /k$ is separable with Galois group $S_4$. It has a unique quadratic subextension with discriminant $(t^2+2t+2)(t^4+2t^3+t+1)$, and thus $k\left(\sqrt{t},\sqrt{t+1},\sqrt{\epsilon_3}\right)^{\Gal}/k$ is geometric. 
        
Applying Chebotarev's density theorem to the class $[\tau]$, we see there are infinitely many irreducible polynomials $D\in\F_3[t]$ satisfying the following:
\begin{enumerate}[label=(\roman*)]
	\item \label{cond: D odd and monic p=3} $D$ is odd degree and monic; 
	\item \label{cond: D inert k(sqrt(t+1)) p=3}$(D)$ is inert in $k(\sqrt{t+1})$;
	\item \label{cond: D inert in k(sqrt(t) and p=3} $(D)$ is inert in $k(\sqrt{t})$;
    \item \label{cond: D split in k(sqrt(ep_3)) p=3} $(D)$ splits completely in $k(\sqrt{\epsilon_3})^{\Gal}$.
\end{enumerate}
By Hensel's Lemma we have that $Dt$ is a square in $k_{\infty}$. By Theorem~\ref{theorem:reciprocity} we have that $D\in k_{t+1}^{\times 2}$ if and only if
\[
1=\left(\frac{D}{t+1}\right)=(-1)^{}\left(\frac{t+1}{D}\right).
\]
Thus condition~\eqref{cond: dt is a square t+1 for p=3} of Lemma~\ref{lem:goodd} is satisfied. Similarly, we have that $D\in k_t^{\times 2}$ if and only if
\[
1=\left(\frac{D}{t}\right)=(-1)\left(\frac{t}{D}\right).
\]
Therefore condition~\eqref{cond: dt is a square t for p=3} of Lemma~\ref{lem:goodd} is satisfied. Thus conditions \ref{cond: D odd and monic p=3},\ref{cond: D inert k(sqrt(t+1)) p=3},\ref{cond: D inert in k(sqrt(t) and p=3} and \ref{cond: D split in k(sqrt(ep_3)) p=3} imply the conditions of Lemma~\ref{lem:goodd} are satisfied, and hence the result holds.
\end{proof}

\begin{example}
Consider the prime $D=t(t+1)^2+1$ and $d=Dt$. It is clear that conditions $\eqref{cond: d is a square infinity for p=3},\eqref{cond: dt is a square t+1 for p=3}$ and $\eqref{cond: dt is a square t for p=3}$ of Lemma \ref{lem:goodd} hold, and one may check condition $\eqref{cond: e_3 is a square at d for p=3}$ in Magma. Thus $X_d$ is a quartic del Pezzo surface over $\F_3(t)$ with no quadratic points.
\end{example}
	
\section{Proof of Theorem~\ref{thm:maindp4}}\label{sec:proofoftheorem3}

\begin{proof}[Proof of Theorem~\ref{thm:maindp4}]
If two quartic del Pezzo surfaces, say $X_d,X_{d'}$, are isomorphic, then they are related by an invertible change of coordinates. If $p>3$, take $d=\frac{D}{\alpha t (t+1)}$ and $d'=\frac{D'}{\alpha t (t+1)}$ for any two distinct primes coming from Lemma~\ref{lem:applyingcheb for not 2,3,7,11} in the case $p\notin\{7,11\}$, and Lemma~\ref{lem: p =3,11 cheb} in the case $p\in\{7,11\}$. In both cases there exist infinitely many, all of which satisfy the hypothesis of Lemma~\ref{prop: under hypothesis BMO to quad point}. If  $p=3$, then let $d=Dt$ and $d'=D't$ for any two distinct primes coming from Lemma~\ref{lem:applyingcheb}, of which there exist infinitely many, all of which satisfy the hypothesis of Lemma~\ref{lem:goodd}.\par
Each of $X_d$ and $X_{d'}$ are contained in exactly two rank 4 quadrics defined over $k(\sqrt{-t})$, with discriminants $D, D'\in k(\sqrt{-t})^{\times}/k(\sqrt{-t})^{\times 2}$ if $p>3$, and discriminants $D(t-1), D'(t-1)$, if $p=3$. Since linear automorphisms of $\PP^4$ preserve discriminants of any quadric hypersurface, it suffices to prove $DD'\notin k(\sqrt{-t})^{\times 2}$.  Indeed, if $DD'\in k(\sqrt{-t})^{\times 2}$, then we would have
\begin{equation*}
	DD'=a^2+2\sqrt{-t}ab-tb^2,
\end{equation*}
for some $a,b\in \OO_{k}$. Since $D,D'\in k^{\times}$, we have that $a=0$ or $b=0$. Since $DD',-DD't\not\in k^{\times2}$, as $D,D'$ are distinct primes, each of these surfaces are non-isomorphic. The result then follows from Lemma~\ref{lem:quadpointrationalline}.
\end{proof}
	
\section{How we chose this family}\label{sec:howwechose}
To find the families, we begin with a generic $Q_0$ and $Q_\infty$ in a similar shape to the one found in \cite{creutz2024quartic}. Namely
\[
M_0=\begin{pmatrix}
		d      & 0      & 0   & 0   & 0   \\ 
		0      & -d\,t  & 0   & 0   & 0   \\ 
		0      & 0      & a_2 & a_3 & 0   \\ 
		0      & 0      & a_3 & 0   & a_4 \\ 
		0      & 0      & 0   & a_4 & 0
    \end{pmatrix},
	\qquad
	M_\infty=\begin{pmatrix}
		0   & -d  & 0   & 0   & 0   \\ 
		-d  & 0   & 0   & 0   & 0   \\ 
		0   & 0   & b_1 & 0   & b_3 \\ 
		0   & 0   & 0   & b_2 & 0   \\ 
		0   & 0   & b_3 & 0   & b_4
	\end{pmatrix}.
	\]
The characteristic polynomial of the associated quartic del Pezzo surface $X$ is
\[
	f(x) \;=\; (x^2 + t)\,\Biggl(
	x^3
	+ \frac{a_2 b_4}{b_1 b_4 - b_3^2}\,x^2
	+ \frac{-a_3^2 b_4 + 2\,a_3 a_4 b_3 - a_4^2 b_1}{b_1 b_2 b_4 - b_2 b_3^2}\,x
	- \frac{a_2 a_4^2}{b_1 b_2 b_4 - b_2 b_3^2}
	\Biggr).
\]
    
Evaluating $M_0+xM_\infty$ at $x=\sqrt{-t}$ gives
\begin{align*}
\epsilon_2&=\det\begin{pmatrix}
			d      & -d\sqrt{-t}      &    & 0   & 0   \\ 
			-d\sqrt{-t}     & -d\,t  & 0   & 0   & 0   \\ 
			0      & 0      & a_2+\sqrt{-t}b_1 & a_3 & \sqrt{-t}b_3   \\ 
			0      & 0      & a_3 & \sqrt{-t}b_2   & a_4 \\ 
			0      & 0      & \sqrt{-t}b_3   & a_4 & \sqrt{-t}b_4
		\end{pmatrix}\\
		&= df_3(\sqrt{-t})\in k(\sqrt{-t})^{\times}/k(\sqrt{-t})^{\times2}.
\end{align*}
By imposing the $x$ coefficient of $f_3$ to be equal to $t$, we ensure $\epsilon_2$ is rational, which makes the Brauer--Manin obstruction easier to compute. In particular, this is equivalent to
\[
\frac{-a_3^2b_4+2a_3a_4b_3-a_4^2b_1}{b_2(b_1b_4-b_3^2)}=t.
\]
	
We also want the local invariant at $t$ to be constantly equal to $1/2$. In order to ensure this, we want $BQ_0+Q_\infty$ to satisfy the hypothesis of Lemma~\ref{lemma2} and the associated conic to not have a $k_t$-point. The pencil of quadrics $BQ_0+Q_{\infty}$ is the orthogonal sum of a rank 2 form and a rank 3 form which diagonalises as 
\[
\beta_1u_1^2+\beta_2u_2^2+\beta_3 u_3^2=0,
\]
where
	
\begin{align*}
\beta_1 &= B a_2 + b_1, \\
\beta_2 &= \frac{-B^2 a_3^2 + B a_2 b_2 + b_1 b_2}{B a_2 + b_1}, \\
\beta_3 &= \frac{B^3 a_2 a_4^2 + B^2 a_3^2 b_4 - 2 B^2 a_3 a_4 b_3 + B^2 a_4^2 b_1- B a_2 b_2 b_4 - b_1 b_2 b_4 + b_2 b_3^2}{B^2 a_3^2 - B a_2 b_2 - b_1 b_2}.
\end{align*}
For the following choice of coefficients:
\[
\begin{aligned}
	a_2 &= t^2 + t^3,\quad & a_3 &= \alpha\,t,\quad & a_4 &= t^2,\\
	b_1 &= \dfrac{\alpha}{t},\quad & b_2 &= t^2,\quad & b_3 &= \alpha,\quad & b_4 &= \dfrac{(3\alpha  - 1)t}{1 + \alpha},
	\end{aligned}
\]
and for $B\in \oo_t$, the Hilbert symbol associated to this conic is
\[
\left(-b_1b_2,-b_1b_4+b_3^2\right)_t=\left(-\alpha t,\alpha(\alpha+1)\right)_t,
\]
which is non-trivial.  

\raggedbottom 
 
\bibliography{thebibliography}

@article{colliot,
  author       = {Colliot--Th{\'e}l{\`e}ne, Jean--Louis},
  title        = {Retour sur l'arithm{\'e}tique des intersections de deux quadriques, avec un appendice par {A}. {K}uznestov},
  journal      = {Journal f{\"u}r die reine und angewandte Mathematik},
  volume       = {806},
  pages        = {147--185},
  year         = {2024},
  doi          = {10.1515/crelle-2023-0081},
  url          = {https://doi.org/10.1515/crelle-2023-0081}
}

@article{creutz2024quartic,
  author  = {Creutz, Brendan and Viray, Bianca},
  title   = {Quartic del {P}ezzo surfaces without quadratic points},
  journal = {Pacific Journal of Mathematics},
  volume  = {337},
  number  = {2},
  pages   = {215--224},
  year    = {2025},
  doi     = {10.2140/pjm.2025.337.215},
  url     = {https://doi.org/10.2140/pjm.2025.337.215}
}

@article{CreutzViray2023,
  author    = {Brendan Creutz and Bianca Viray},
  title     = {Quadratic points on intersections of two quadrics},
  journal   = {Algebra \& Number Theory},
  volume    = {17},
  number    = {8},
  pages     = {1411--1451},
  year      = {2023},
  doi       = {10.2140/ant.2023.17.1411},
  url       = {https://msp.org/ant/2023/17-8/ant-v17-n8-p03-p.pdf}
}

@book{ColliotThelene2021,
  author    = {Jean--Louis Colliot--Th{\'e}l{\`e}ne and Alexei N. Skorobogatov},
  title     = {The Brauer--Grothendieck Group},
  year      = {2021},
  publisher = {Springer},
  address   = {Cham},
  series    = {Ergebnisse der Mathematik und ihrer Grenzgebiete. 3. Folge / A Series of Modern Surveys in Mathematics},
  volume    = {71},
  isbn      = {978-3-030-74247-8},
  doi       = {10.1007/978-3-030-74248-5},
  url       = {https://link.springer.com/book/10.1007/978-3-030-74248-5}
}

@book{milne1980etale,
  author    = {Milne, James S.},
  title     = {\'Etale Cohomology},
  series    = {Princeton Mathematical Series},
  volume    = {33},
  year      = {1980},
  publisher = {Princeton University Press},
  address   = {Princeton, NJ},
  isbn      = {0691171106}
}

@book{rosen,
author = {Rosen, Michael},
address = {New York, NY},
booktitle = {Number Theory in Function Fields},
edition = {1st ed. 2002.},
isbn = {1-4757-6046-9},
language = {eng},
publisher = {Springer New York},
series = {Graduate Texts in Mathematics, 210},
title = {Number Theory in Function Fields },
year = {2002},
}

@phdthesis{Amer1976,
  author       = {Amer, Mahmoud},
  title        = {Quadratische formen über funktionenkörpern},
  school       = {Johannes Gutenberg Universität, Mainz},
  year         = {1976},
  note         = {Unpublished dissertation}
}

@phdthesis{reid1972complete,
  author       = {Reid, Miles},
  title        = {The complete intersection of two or more quadrics},
  school       = {Trinity College, Cambridge},
  year         = {1972},
  month        = {June},
  type         = {PhD thesis},
  url          = {https://mreid.warwick.ac.uk/3folds/qu.pdf},
}

@article{Brumer1978,
  author       = {Armand Brumer},
  title        = {Remarques sur les couples de formes quadratiques},
  journal      = {Comptes Rendus de l'Acad\'emie des Sciences. S\'erie A-B},
  year         = {1978},
  volume       = {286},
  number       = {16},
  pages        = {A679--A681}
}

@book{Wittenberg2007,
  author    = {Olivier Wittenberg},
  title     = {Intersections de deux quadriques et pinceaux de courbes de genre 1},
  series    = {Lecture Notes in Mathematics},
  volume    = {1901},
  publisher = {Springer},
  address   = {Berlin},
  year      = {2007},
  language  = {French},
  mrnumber  = {2307807}
}

@article{Odoni1985GaloisIterates,
  author  = {Odoni, Robert W.K.},
  title   = {The {G}alois theory of iterates and composites of polynomials},
  journal = {Proceedings of the London Mathematical Society},
  series  = {Series 3},
  volume  = {51},
  number  = {3},
  pages   = {385--414},
  month   = {Nov},
  year    = {1985},
  doi     = {10.1112/plms/s3-51.3.385},
}

@misc{GroupNames,
    title = {Group{N}ames database},
    author = {Dokchitser, Tim},
    url     = {https://people.maths.bris.ac.uk/~matyd/GroupNames/1/e3/S3byC2%5E2.html#s1},
    note = {Accessed: 12-02-2026}
}

@article{SPRINGER1956238,
title = {Quadratic forms over Fields with a Discrete Valuation {II}. {N}orms},
journal = {Indagationes Mathematicae (Proceedings)},
volume = {59},
pages = {238-246},
year = {1956},
issn = {1385-7258},
doi = {https://doi.org/10.1016/S1385-7258(56)50031-5},
url = {https://www.sciencedirect.com/science/article/pii/S1385725856500315},
author = {Springer, Tonny Albert}
}

@article{Tian2025LocalGlobal,
  author    = {Tian, Zhiyu},
  title     = {Local-global principle and integral {T}ate conjecture for certain varieties},
  journal   = {Journal of the American Mathematical Society},
  volume    = {38},
  number    = {3},
  year      = {2025},
  pages     = {703--782},
  doi       = {10.1090/jams/1054},
  publisher = {American Mathematical Society}
}
	
\bibliographystyle{alpha}

\end{document}